\documentclass{amsart}

\usepackage{graphicx} 
\usepackage{times}
\usepackage{graphicx}
\usepackage[T1]{fontenc}
\usepackage[utf8]{inputenc}
\usepackage[english]{babel}
\usepackage{indentfirst}
\usepackage{fancyhdr}
\usepackage{amsthm}
\usepackage{amssymb}
\usepackage{amsfonts}
\usepackage{amsmath}
\usepackage{latexsym}
\usepackage{mathrsfs}
\usepackage{setspace}
\usepackage{xcolor}
\usepackage{enumitem}
\DeclareGraphicsExtensions{.png}

\usepackage[nameinlink, capitalize]{cleveref}

\theoremstyle{plain}

\newtheorem{theorem}{Theorem}[section]
\newtheorem{lemma}[theorem]{Lemma}
\newtheorem{proposition}[theorem]{Proposition}
\newtheorem{corollary}[theorem]{Corollary}
\newtheorem*{theoremA}{Theorem A}
\newtheorem*{theoremB}{Theorem B}
\newtheorem*{theoremC}{Theorem C}

\theoremstyle{definition}

\newtheorem{definition}[theorem]{Definition}

\newtheorem{example}[theorem]{Example}

\DeclareMathOperator*{\Dr}{Dr}
\DeclareMathOperator*{\Cr}{Cr}

\newcommand{\N}{\mathbb{N}}
\renewcommand{\P}{\mathbb{P}}
\newcommand{\Z}{\mathbb{Z}}

\newcommand{\CC}{\mathcal{C}}
\newcommand{\LL}{\mathcal{L}}

\newcommand{\UU}{\mathcal{U}}

\title{Profinite groups with complemented closed subgroups}

\author{Gustavo A.\ Fern\'andez-Alcober}
\address{Gustavo A. Fern\'{a}ndez-Alcober: Department of Mathematics, University of the Basque Country UPV/EHU, 48080 Bilbao, Spain}
\email{gustavo.fernandez@ehu.eus}

\author{Giulia Sabatino}
\address{Giulia Sabatino: Department of Mathematics, University of the Basque Country UPV/EHU, 48080 Bilbao, Spain}
\email{giul.sabatino@gmail.com}

\thanks{\\[-5pt] Both authors are supported by the Spanish Government, grant PID2020-117281GB-I00, partly with FEDER funds.
The first author is also supported by the Basque Government, grant IT483-22.
The second author is a member of GNSAGA (INdAM) and AGTA-Advances in
Group Theory and Applications (http://www.advgrouptheory.com).}

\keywords{Profinite groups, complemented subgroups, closed subgroups}
\subjclass[2020]{20E18, 20E15}

\date{}

\begin{document}

\begin{abstract}
A group $G$ is said to be a $C$-group if every subgroup $H$ has a permutable complement, i.e.\ if there exists a subgroup
$K$ of $G$ such that $G=HK$ and $H \cap K=1$.
In this paper, we study the profinite counterpart of this concept.
We say that a profinite group $G$ is profinite-$C$ if every closed subgroup admits a closed permutable complement.
We first give some equivalent variants of this condition and then we determine the structure of profinite-$C$ groups:
they are the semidirect products $G=B\ltimes A$ of two closed subgroups $A=\Cr_{i\in I} \, \langle a_i \rangle$ and
$B=\Cr_{j\in J} \, \langle b_j \rangle$ that are cartesian products of cyclic groups of prime order, and with every
$\langle a_i \rangle$ normal in $G$.
Finally, we show that a profinite-$C$ group is a $C$-group if and only if it is torsion and
$|G:Z(G)\overline{G'}|<\infty$.
\end{abstract}

\maketitle

\section{Introduction}

Given a group $G$ and a subgroup $H$ of $G$, a \emph{complement} of $H$ in $G$ is a subgroup $K$ such that $G=\langle H,K \rangle$ and $H\cap K=1$, and $K$ is a \emph{permutable complement} of $H$ in $G$ if we further have $G=HK$.
(Permutable complements are called just complements in some textbooks in group theory; here we stick to the terminology
used in R.\ Schmidt's book \cite{Sch}.)
We say that $G$ is a \emph{$C$-group} if every subgroup of $G$ has a permutable complement.
Finite $C$-groups were studied by Philip Hall \cite{H}, who characterized them as the groups that are isomorphic
to a subgroup of the direct product of two groups of squarefree order.
Later on, N.V.\ \v{C}ernikova \cite{C} showed that an arbitrary group $G$ is a $C$-group if and only if
$G=B\ltimes A$ is a semidirect product, where $B=\Dr_{j\in J} \, \langle b_j \rangle$ and $A=\Dr_{i\in I} \, \langle a_i \rangle$
are both direct products of cyclic groups of prime order and $\langle a_i \rangle \trianglelefteq G$ for every $i\in I$.
The reader can find a proof of this result, together with other relevant facts regarding $C$-groups in Section 3.2
of R.\ Schmidt's book \cite{Sch}.

In the realm of profinite groups, closed subgroups play a special role, being those subgroups that are again profinite.
For this reason, many problems regarding subgroups in profinite groups are formulated by considering only closed subgroups.
In this paper, we address the following question: what are the profinite groups in which every \emph{closed} subgroup admits a permutable complement?

A related question is this: if every closed subgroup admits a permutable complement, does it also admit a \emph{closed} permutable complement?
For example, let $p$ be a fixed prime and let $G$ be the cartesian product of countably many copies of the cyclic group of order
$p$.
It is well-known that $G$ has maximal subgroups that are not open.
If $M$ is such a subgroup then $M$ has a permutable complement $H$ of order $p$ in $G$.
Hence $H$ is a closed subgroup of $G$ having a non-closed permutable complement.
This does not imply, of course, that $H$ cannot have a closed permutable complement.

Our first theorem clarifies the situation.

\begin{theoremA}
\label{theoremA}
Let $G$ be a profinite group.
Then the following are equivalent:
\begin{enumerate}
\item
Every closed subgroup of $G$ has a closed permutable complement.
\item
Every closed subgroup of $G$ has a permutable complement.
\item
Every open subgroup of $G$ has a permutable complement.
\item
Every open subgroup of $G$ has a closed permutable complement.
\item
If $\mathcal{N}$ is a fundamental system of neighborhoods of the identity consisting of open
normal subgroups of $G$, then $G/N$ is a $C$-group for every $N\in\mathcal{N}$.
\item
$G$ is an inverse limit of finite $C$-groups.
\end{enumerate}
\end{theoremA}

If a profinite group $G$ satisfies condition (i) in Theorem A (and so all of the other equivalent conditions),
we say that $G$ is a \emph{profinite-$C$ group}.
Our next goal is to describe the structure of profinite-$C$ groups.
To this purpose, it is convenient to define when a profinite group $G$ is an \emph{internal cartesian product} of a family
$\{H_i\}_{i\in I}$ of its subgroups (see \cref{dfn internal cartesian} below).
Since such an internal cartesian product will be isomorphic to the external cartesian product $\Cr_{i\in I} \, H_i$,
we write $G=\Cr_{i\in I} \, H_i$ also for the internal product.

\begin{theoremB}
\label{theoremB}
Let $G$ be a profinite group.
Then the following are equivalent:
\begin{enumerate}
\item
$G$ is a profinite-$C$ group.
\item
$G=B\ltimes A$, where $B=\Cr_{j\in J} \, \langle b_j \rangle$ and $A=\Cr_{i\in I} \, \langle a_i \rangle$ are closed subgroups of $G$, all elements $b_j$ and $a_i$ are of prime order, and
$\langle a_i \rangle \trianglelefteq G$ for every $i\in I$.
\end{enumerate}
Furthermore, if these properties hold then $A$ can be chosen to be $\overline{G'}$, the topological closure
of $G'$.
\end{theoremB}

All $C$-groups are torsion groups, but as Theorem B implies, this need not be the case for a profinite-$C$ group: consider for example the cartesian product $\Cr_{p\in\P} \, C_p$, where $\P$ is the set of all primes.
This raises the question of whether a profinite $C$-group is a $C$-group if and only if it is torsion.
As we next see, being torsion is not usually enough and we need an extra condition.

\begin{theoremC}
Let $G$ be a profinite-$C$ group.
Then the following are equivalent:
\begin{enumerate}
\item
$G$ is a $C$-group.
\item
$G$ is torsion and $|G:Z(G)\overline{G'}|<\infty$.
\end{enumerate}
Thus an abelian profinite-$C$ group is a $C$-group if and only if it is torsion.
\end{theoremC}

\noindent
\textit{Notation.\/}
If $G$ is a profinite group, we write $H\le_c G$ or $H\le_{o} G$ to mean that $H$ is a closed or open subgroup of $G$,
respectively.
Also, we let $\pi(G)$ be the set of prime divisors of the order of $G$, which is a supernatural number.
If $\{G_i\}_{i\in I}$ is a family of groups, we denote by $\Cr_{i\in I} \, G_i$ the full cartesian product of the groups
$G_i$, and by $\Dr_{i\in I} \, G_i$ the direct product of the $G_i$, i.e. the set of tuples of finite support.

\section{Proof of the results}

As mentioned in the introduction, we define a profinite-$C$ group as a profinite group in which all closed
subgroups have a closed permutable complement.
To start with, we collect some elementary properties of profinite-$C$ groups in the next three lemmas.
We omit the proofs, since they are either straightforward or follow the proofs of similar results for $C$-groups,
as one can see in Section 3.2 of the book \cite{Sch}.
Given a group $G$ and $H\le G$, we say that $K\le G$ is a \emph{supplement} of $H$ in $G$ if $G=HK$.

\begin{lemma}
\label{subgroups and quotients of profinite-C}
Let $G$ be a profinite group and let $J\le_{c} G$ and $N\trianglelefteq_{c} G$.
Then:
\begin{enumerate}
\item
If $H\le_{c} J$ and $H$ has a closed permutable complement $K$ in $G$, then $J\cap K$ is a closed permutable
complement of $H$ in $J$.
In particular, if $G$ is a profinite-$C$ group, then $J$ is also profinite-$C$.
\item
If $N\le H\le_{c} G$ and $H$ has a closed permutable complement $K$ in $G$, then $KN/N$ is a closed permutable
complement of $H/N$ in $G/N$.
In particular, if $G$ is a profinite-$C$ group then also $G/N$ is profinite-$C$.
\end{enumerate}
\end{lemma}

\begin{lemma}
\label{supplements and complements}
Let $G$ be a profinite-$C$ group and let $H\le_{c} G$ and $N\trianglelefteq_{c} G$.
Then:
\begin{enumerate}
\item 
If $S$ is a closed supplement of $H$ in $G$, then $S$ contains a closed permutable complement
of $H$ in $G$.
\item
If $S/N$ is a closed permutable complement of $HN/N$ in $G/N$ then there exists
a closed permutable complement $K$ of $H$ in $G$ such that $S=KN$.
\end{enumerate}
\end{lemma}

\begin{lemma}
\label{direct product of profinite-C}
Let $G_1$ and $G_2$ be two profinite-$C$ groups.
Then $G_1\times G_2$ is also profinite-$C$.
\end{lemma}

Obviously, Lemma \ref{direct product of profinite-C} can be extended to the cartesian product of a finite number of
profinite-$C$ groups.
The proof of Theorem A relies on generalising it to arbitrary cartesian products.

\begin{theorem}
\label{cartesian product of profinite-C}
Let $\{G_i\}_{i\in I}$ be a family of profinite-C groups.
Then the cartesian product $\Cr_{i\in I} \, G_i$ is also a profinite-$C$ group.
As a consequence, an inverse limit of profinite-$C$ groups is profinite-$C$.
\end{theorem}

\begin{proof}
Set $G=\Cr_{i \in I} \, G_i$.
We may assume that $I$ is well-ordered and so also that $I$ is an ordinal number $\alpha$.
Thus $I=\{ \mu \mid \text{$\mu$ is an ordinal and $\mu<\alpha$} \}$.
For every $\mu\in I$, consider the cartesian product
$N_\mu=\Cr_{i \in I} \, H_i$, where $H_i$ is trivial if $i=\mu$ and $H_i=G_i$ otherwise.
Then $N_\mu$ is a closed normal subgroup of $G$ and $G/N_{\mu}\cong G_{\mu}$ is a profinite-$C$ group.
Let us define
\[
I_{\mu} = \bigcap_{\lambda<\mu} \, N_{\lambda}
\]
for $1\le \mu \le \alpha$.
We claim that $G/I_{\mu}$ is a profinite-$C$ group, which proves the theorem, since $I_{\alpha}=1$.
More precisely, we are going to show that, for every closed subgroup $H$ of $G$, its image $HI_{\mu}/I_{\mu}$
has a permutable complement $K_{\mu}/I_{\mu}$ in $G/I_{\mu}$, and that these complements can be chosen to form a
descending chain: $K_{\mu}\le K_{\lambda}$ for all $\mu\ge \lambda$.

We use transfinite induction.
If $\mu=1$ then $I_1=N_0$ and the result holds, so we assume $\mu>1$.
Suppose first that $\mu$ is a successor ordinal.
Then $I_{\mu}=I_{\mu-1}\cap N_{\mu-1}$ and $G/I_{\mu}$ embeds as a closed subgroup in $G/I_{\mu-1}\times G/N_{\mu-1}$.
From the induction hypothesis, together with  \cref{direct product of profinite-C} and (i) of
\cref{subgroups and quotients of profinite-C}, $G/I_{\mu}$ is a profinite-$C$ group.
Now if $H$ is closed in $G$, again by induction there exists a  permutable complement $K_{\mu-1}/I_{\mu-1}$ of
$HI_{\mu-1}/I_{\mu-1}$ in $G/I_{\mu-1}$, with $K_{\mu-1}$ closed in $G$.
By (ii) of \cref{supplements and complements}, there exists a permutable complement $K_{\mu}/I_{\mu}$
of $HI_{\mu}/I_{\mu}$ in $G/I_{\mu}$ that is contained in $K_{\mu-1}/I_{\mu}$.
Thus $K_{\mu}\le K_{\mu-1}$ and the induction is complete in this case.

Assume now that $\mu$ is a limit ordinal.
Let $H$ be a closed subgroup of $G$.
By induction, there exists a descending chain $\{K_{\lambda}\}_{\lambda<\mu}$ of closed subgroups of $G$ such
that $G=HK_{\lambda}$ and $H\cap K_{\lambda}\le I_{\lambda}$ for every $\lambda<\mu$.
We define
\[
K_{\mu} = \bigcap_{\lambda<\mu} \, K_{\lambda}.
\]
Then
\[
H\cap K_{\mu} = \bigcap_{\lambda<\mu} \, (H\cap K_{\lambda}) \le \bigcap_{\lambda<\mu} I_{\lambda} = I_{\mu},
\]
since $\mu=\cup_{\lambda<\mu} \, \lambda$.
On the other hand, Proposition 2.1.4 (a) of \cite{RZ} yields that
\[
G = \bigcap_{\lambda<\mu} \, HK_{\lambda} = H \big( \bigcap_{\lambda<\mu} \, K_{\lambda} \big) =HK_{\mu}. 
\]
This completes the induction.

The final assertion follows immediately from (i) of \cref{subgroups and quotients of profinite-C}, since an inverse limit of groups
is a closed subgroup of the cartesian product of the groups in the inverse limit.
\end{proof}

\begin{corollary}
\label{completion of C-group is profinite-C}
If $G$ is a $C$-group, then the profinite completion of $G$ is a profinite-$C$ group.
\end{corollary}

Now we can easily prove Theorem A.

\begin{proof}[Proof of Theorem A]
The implications (i)$\Rightarrow$(ii)$\Rightarrow$(iii) and (v)$\Rightarrow$(vi) are trivial.
Since open subgroups have finite index in a profinite group, a permutable complement of an open subgroup must
be finite, and in particular closed.
This proves that (iii) implies (iv).
Note that (iv)$\Rightarrow$(v) follows from (ii) of \cref{subgroups and quotients of profinite-C}.
Finally, the implication (vi)$\Rightarrow$(i) is a special case of \cref{cartesian product of profinite-C}.
\end{proof}

Since every $C$-group is metabelian, being the semidirect product of two abelian groups, we immediately
obtain the following corollary.

\begin{corollary}
\label{profinite-C is metabelian}
Let $G$ be a profinite-$C$ group.
Then $G$ is metabelian.
\end{corollary}

Let us now proceed towards the proof of Theorem B.
A key step will be the determination of the structure of abelian closed normal subgroups of profinite-$C$ groups,
which we describe in \cref{structure of closed abelian normal} below.
We first need the following result.

\begin{lemma}
\label{abelian normal sbgp}
Let $G$ be a profinite-$C$ group and let $A$ be an abelian closed normal subgroup of $G$.
Then the following hold:
\begin{enumerate}
\item 
For every $B\le_{c} A$, we have $A=B\times L$ for some $L\trianglelefteq_{c} G$.
\item
If $A$ is a minimal closed normal subgroup of $G$ then $A$ is of prime order.
\item 
If $A\ne 1$ and $B$ is maximal among the open subgroups of $A$ that are normal in $G$, then
$A/B$ is isomorphic as a $G$-module to a minimal closed normal subgroup of $G$ contained in $A$.
In particular, $A$ contains minimal closed normal subgroups of $G$. 
\end{enumerate}
\end{lemma}

\begin{proof}
(i)
Let $K$ be a closed permutable complement of $B$ in $G$ and set $L=A\cap K$.
By (i) of \cref{subgroups and quotients of profinite-C}, we have $A=B\times L$.
Observe that $L\trianglelefteq A$, since $A$ is abelian, and that $L\trianglelefteq K$, since
$A\trianglelefteq G$.
As $G=AK$, we conclude that $L\trianglelefteq G$.

(ii)
By (i), every closed subgroup $B$ of $A$ provides a closed subgroup of $A$ that is normal in $G$ and avoids $B$.
From the minimality assumption on $A$, it follows that the only closed subgroups of $A$ are the trivial subgroup and $A$ itself.
On the other hand, since $A$ is profinite, the intersection of its open normal subgroups is trivial.
It readily follows that $A$ is of prime order.

(iii)
By (i), we have $A=B\times L$ for some $L\trianglelefteq_{c} G$.
Consequently we have a $G$-module isomorphism $L\cong A/B$, and since $A/B$ is a simple $G$-module by the
choice of $B$, it follows that $L$ is a minimal normal subgroup of $G$.
\end{proof}

Note that, by \cref{profinite-C is metabelian}, a minimal closed normal subgroup of a profinite-$C$ group is
necessarily abelian, so of prime order by the previous lemma.
Compare this with the fact that, in an arbitrary profinite group $G$, a minimal closed normal subgroup $D$ of $G$ (if it exists)
is of finite order, since there must be an open normal subgroup $N$ of $G$ such that $D\cap N=1$.

\vspace{10pt}

We will use the last lemma to show that an abelian closed normal subgroup $A$ of a profinite-$C$ group $G$
is isomorphic to a cartesian product of some of its subgroups.
Since we also want to provide information about the embedding of this cartesian product inside $G$, it is more convenient
to avoid isomorphisms and speak in internal terms, describing $A$ by means of some of its subgroups (which will turn out to be normal in $G$).
For this reason, we give the following definition.

\begin{definition}
\label{dfn internal cartesian}
A profinite group $G$ is said to be the \emph{internal cartesian product} of a family $\{H_{i}\}_{i\in I}$ of its subgroups
if the following conditions hold:
\begin{enumerate}[itemsep=0.5em]
\item
$H_i \unlhd_{c} G$ for every $i \in I$.
\item
$G=\overline{\langle H_{i}\,| \, \in I\rangle}$.
\item
$\bigcap_{i\in I} \, \overline{\langle H_{j} \mid j\neq i \rangle}=1$.
\end{enumerate}
\end{definition}

With this definition, it is obvious that any external cartesian product $\Cr_{i\in I} \, H_i$ is an internal cartesian
product of subgroups isomorphic to the $H_i$.
On the other hand, assume that $G$ is the internal cartesian product of a family of
subgroups $\{H_i\}_{i\in I}$.
Then, by Lemma 2.4.2 of \cite{W}, $G$ is isomorphic to the external cartesian product $\Cr_{i\in I} \, H_i$.
Actually, the isomorphism obtained in the proof of that lemma sends every $h_i\in H_i$ to the tuple with
$h_i$ in position $i$ and $1$ in every other position.
For this reason, with a slight abuse of notation and following the usual convention for direct products,
we write $G=\Cr_{i\in I} \, H_i$ in this case, the same as for the external cartesian product.

\vspace{10pt}

After this digression, we need a final remark before proceeding to the determination of the structure of abelian
closed normal subgroups in profinite-$C$ groups.
Let $G$ be a profinite group and let $M$ be a non-zero profinite $G$-module.
Following \cite[p.\ 124]{W}, we define the \textit{radical} $\Phi(M)$ of $M$ as the intersection of
all maximal open $G$-submodules of $M$.
By Proposition 7.4.5 of \cite{W}, the condition $\Phi(M)=0$ implies that $M$ is topologically isomorphic to a
cartesian product of simple profinite $G$-modules.
Relying on this fact, we can prove the following lemma.

\begin{lemma}
\label{structure of closed abelian normal}
Let G be a profinite-$C$ group and let A be a non-trivial abelian closed normal subgroup of G.
Then A is an internal cartesian product of $G$-invariant subgroups of $A$ of prime order.
\end{lemma}

\begin{proof}
We view $A$ as a profinite $G$-module by conjugation.
If we show that $\Phi(A)$ is the trivial subgroup then, by the result mentioned before the lemma,
there is a topological $G$-module isomorphism $\varphi:C\rightarrow A$, where $C=\Cr_{i\in I} \, C_i$ is a
cartesian product of simple profinite $G$-modules.
If we set $A_i=\varphi(C_i)$ for $i\in I$, then every $A_i$ is a minimal closed $G$-invariant subgroup of $A$
and $A$ is the internal cartesian product of the family $\{A_i\}_{i\in I}$.
Now by (ii) of \cref{abelian normal sbgp}, each $A_i$ is of prime order and we are done.

Thus it suffices to show that $\Phi(A)=1$.
Suppose, by way of contradiction, that $\Phi(A)\ne 1$.
We are going to make use of the results stated in \cref{abelian normal sbgp}.
On the one hand, since $\Phi(A)$ is closed and normal in $G$, it contains a minimal closed normal
subgroup of $G$, say $B$.
We can then decompose $A=B\times L$, where $L\trianglelefteq_{c} G$.
Since $B$ is of prime order, $L$ is also a maximal open subgroup of $A$.
It follows that $B\le \Phi(A)\le L$, which is a contradiction with the product of $B$ and $L$ being direct.
\end{proof}

The characterisation of abelian profinite-$C$ groups follows immediately from
\cref{cartesian product of profinite-C} and \cref{structure of closed abelian normal}.

\begin{corollary}
\label{characterisation abelian profinite-C}
Let $G$ be an abelian profinite group.
Then $G$ is a profinite-$C$ group if and only if $G$ is an internal cartesian product of subgroups
of prime order.
\end{corollary}

Contrary to the case of direct products, a semidirect product of two profinite-$C$ groups need not be
profinite-$C$.
For example, consider a non-abelian finite $p$-group $P$ of order $p^3$ and exponent $p$, where $p$ is an
odd prime.
Then $P$ can be decomposed as a semidirect product of an elementary abelian $p$-group of order $p^2$ and a cyclic group of order $p$, both of which are $C$-groups.
However, it is clear that the center of $G$ does not have a permutable complement in $G$.
Our next lemma provides a sufficient condition for a semidirect product of profinite-$C$ groups to be profinite-$C$.

\begin{lemma}
\label{semidirect profinite-C}
Let $G$ be a profinite group and assume that $G=H\ltimes N$ is the semidirect product of two closed subgroups $H$ and $N$.
If $H$ is a profinite-$C$ group and every open subgroup of $N$ has a $G$-invariant permutable complement in $N$, then $G$ is also profinite-$C$.
\end{lemma}

\begin{proof}
By Theorem A, it suffices to see that every open subgroup $E$ of $G$ has a permutable complement in $G$.
By assumption, $E\cap N$ has a $G$-invariant permutable complement $C$ in $N$.
Thus $N=(E\cap N)C$ and $E\cap C=E \cap N \cap C=1$.
On the other hand, since $H$ is profinite-$C$,
$EN\cap H$ has a permutable complement $D$ in $H$.
Hence $H=(EN \cap H)D$ and $EN\cap D=EN\cap H \cap D=1$.
Now, since $C$ is $G$-invariant, the product $CD$ is a subgroup of $G$, and
\[
ECD = E(E\cap N)CD = END = EN(EN\cap H)D = ENH = G
\]
and, since
\[
E\cap CD \le EN\cap CD = C(EN\cap D) = C,
\]
we also have
\[
E\cap CD \le E\cap C = 1.
\]
Thus $CD$ is a permutable complement of $E$ in $G$.
\end{proof}

Now we can readily prove Theorem B, which characterises profinite-$C$ groups in all generality.

\begin{proof}[Proof of Theorem B]
(i)$\Rightarrow$(ii).
By \cref{profinite-C is metabelian}, $G$ is metabelian.
Thus $\overline{G'}$ is an abelian closed normal subgroup of $G$ and, by
\cref{structure of closed abelian normal}, it is an internal cartesian product of groups of prime order,
all of which are normal in $G$.
Let $B$ be a closed permutable complement of $\overline{G'}$ in $G$.
Then $B\cong G/\overline{G'}$ is an abelian profinite-$C$ group, and consequently also a cartesian product of groups of prime order by \cref{characterisation abelian profinite-C}.

(ii)$\Rightarrow$(i).
By \cref{characterisation abelian profinite-C} and \cref{semidirect profinite-C}, it suffices to show that every open subgroup $H$ of $A$ has a
$G$-invariant permutable complement in $A$.
We consider the set
\[
\LL = \{ K\le A \mid \text{$H\cap K=1$ and $K\trianglelefteq G$} \}.
\]
Then every chain in $\LL$ with respect to inclusion has clearly an upper bound in $\LL$, and by
Zorn's Lemma, $\LL$ has a maximal element $L$.

We claim that $A=HL$, which shows that $L$ is a $G$-invariant permutable complement of $H$ in $A$,
as desired.
By way of contradiction, assume that $HL$ is properly contained in $A$.
As $H$ is an open subgroup of $A=\Cr_{i\in I} \, \langle a_i \rangle$, it contains a subgroup of the
form $N=\Cr_{i\in I} \, N_i$, where $N_i=\langle a_i \rangle$ for every $i$ outside a finite subset
$S$ of $I$.
Since $N\le HL<A$, we must have $a_s\not\in HL$ for some $s\in S$.
Now, by assumption, $\langle a_s \rangle$ is a normal subgroup of $G$ of prime order.
It follows that $\langle a_s \rangle L\trianglelefteq G$ and also that $HL\cap \langle a_s \rangle=1$.
The latter condition, together with $H\cap L=1$, implies that $H\cap \langle a_s \rangle L=1$.
Thus we obtain a contradiction with the maximality of $L$ in the set $\LL$, since $a_s\not\in L$.
This proves the claim and completes the proof of the theorem.
\end{proof}

As mentioned in the introduction, there exist profinite-$C$ groups with elements of infinite order and,
as a consequence, a profinite-$C$ group need not be a $C$-group.
The following example shows that even a torsion profinite-$C$ group need not be a $C$-group.

\begin{example}
Let $p$ and $q$ be prime numbers such that $p\equiv 1\pmod q$.
For every $n\in\N$, let $G_n=\langle x_n,y_n \rangle$ be a non-abelian group of order $pq$, where $x_n$ is of order $p$, $y_n$ is of order $q$, and $\langle x_n \rangle \trianglelefteq G_n$.
Then $G_n$ is a $C$-group and, by \cref{cartesian product of profinite-C}, the cartesian product $G=\Cr_{n\in\N} \, G_n$ is a profinite-$C$ group.
Also, $G$ has finite exponent $pq$.
Assume, by way of contradiction, that $G$ is a $C$-group.
By Lemmas 3.1.7 and 3.2.3 of \cite{Sch}, the abelian normal subgroup $A=\Cr_{n\in\N} \, \langle x_n \rangle$ of $G$ is a direct product of minimal normal subgroups of $G$, and these are cyclic of prime order.
We claim that the only cyclic subgroups of $A$ that are normal in $G$ are the $\langle x_n \rangle$ for $n\in\N$.
It then follows that $A$ is a direct product of countably many cyclic subgroups and, as a consequence, $A$ is countable.
This is a contradiction.

Let us prove the claim of the previous paragraph.
Suppose that $g=(g_n)_{n\in\N}$ is an element of $A$ with at least two non-trivial components, say $g_m$ and $g_n$, and let
$y$ be the tuple having $y_n$ in the $n$th component and $1$ elsewhere.
Then the conjugate $g^y$ has $g_m$ in the $m$th component and a power of $g_n$ different from $g_n$ in the $n$th
component.
It follows that $g^y$ cannot be a power of $g$ and consequently $\langle g \rangle \not\trianglelefteq G$, as desired.
\end{example}

Thus the following question arises: what torsion profinite-$C$ groups are $C$-groups?
Our final goal is to prove Theorem C, which gives an answer to this question.
We start by showing two cases in which a torsion profinite-$C$ group is a $C$-group, namely
abelian groups and groups in which the closure $\overline{G'}$ of the derived subgroup is open in $G$.

\begin{proposition}
\label{torsion abelian}
Let $G$ be an abelian profinite-$C$ group.
Then the following are equivalent:
\begin{enumerate}
\item 
$G$ is a $C$-group.
\item
$G$ is a torsion group.
\item
$G$ has squarefree finite exponent.
\end{enumerate}
\end{proposition}

\begin{proof}
First of all, since $G$ is an abelian profinite-$C$ group, by \cref{characterisation abelian profinite-C} we can write 
$G=\Cr_{i\in I} \, \langle x_i \rangle$, where every $x_i$ is of prime order, say $p_i$.
As a consequence, $G$ is a torsion group if and only if the set $\{p_i\mid i\in I\}$ is finite, i.e.\ if
and only if $G$ has (squarefree) finite exponent.
Since $C$-groups are torsion, this proves in particular that (i) implies (iii).
Let us now see that (iii) implies (i).
Since $G$ is abelian of squarefree finite exponent then, by \cite[4.3.5]{R}, $G$ can be expressed as a direct product
of cyclic groups, each of squarefree order. 
Then $G$ can be further decomposed into a direct product of cyclic subgroups of prime order, and so
$G$ is a $C$-group by \v{C}ernikova's result in \cite{C}.
\end{proof}

\begin{proposition}
\label{torsion G' open}
Let $G$ be a profinite-$C$ group such that $\overline{G'}$ is an open subgroup of $G$.
Then $G$ is a $C$-group if and only if $G$ is a torsion group.
\end{proposition}

\begin{proof}
Assume that $G$ is profinite-$C$ and torsion, and that $\overline{G'}$ is open.
By Theorem B, we can write $G=K\ltimes \overline{G'}$ for some finite subgroup $K$ of $G$ that is a direct
product of groups of prime order.
In order to prove that $G$ is a $C$-group, it suffices to see that $\overline{G'}$ is a direct product of
$G$-invariant cyclic subgroups of prime order.

By \cite[Proposition 2.4.3]{W}, we can decompose the abelian group $\overline{G'}$ as the cartesian product of
its Sylow pro-$p$ subgroups for $p\in\pi(G)$.
Since $\overline{G'}$ is torsion, the set $\pi(G)$ is finite, and we actually have a direct product decomposition.
Hence it suffices to see that, for every $p\in\pi(G)$, the Sylow pro-$p$ subgroup $P$ of $G$ decomposes
as a direct product of $G$-invariant subgroups of order $p$.
By \cref{structure of closed abelian normal}, we can write $P=\Cr_{i\in I} \, \langle x_i \rangle$ as an internal
cartesian product, where each $\langle x_i \rangle$ is of order $p$ and normal in $G$.

For every $i\in I$, the action of $K$ on $\langle x_i \rangle$ by conjugation induces a group homomorphism
$\theta_i:K\longrightarrow \UU(\Z/p\Z)$, given by
\[
x_i^k = x_i^{\theta_i(k)}, \quad \text{for every $k\in K$.}
\]
Now we define an equivalence relation $\sim$ on $I$ by letting $i\sim j$ when $\theta_i=\theta_j$.
Since $K$ is finite, the set $\CC$ of equivalence classes for $\sim$ in $I$ is finite.
For every $C\in\CC$, let
\[
P_C = \Cr_{i\in C} \, \langle x_i \rangle.
\]
Then we have
\begin{equation}
\label{direct product of the P_C}
P = \Dr_{C\in\CC} \, P_C.
\end{equation}

On the other hand, for each choice of $C\in\CC$ and $k\in K$, let $\lambda$ be the common value of $\theta_i(k)$
as $i$ runs over $C$.
Then for every $g\in P_C$ we have $g^k=g^{\lambda}$, and consequently $\langle g \rangle\trianglelefteq G$.
Since $P_C$ is an elementary abelian $p$-group, it follows that we can write $P_C$ as the direct product
of normal subgroups of $G$ of order $p$.
By \eqref{direct product of the P_C}, the same is true for $P$, which completes the proof.
\end{proof}

\begin{proof}[Proof of Theorem C]
First of all, since $G$ is profinite-$C$, we can write $G=K\ltimes \overline{G'}$ as in Theorem B.
Since $K$ and $\overline{G'}$ are both abelian, we have $Z(G)=(Z(G)\cap K)\times (Z(G)\cap \overline{G'})$.
Consequently
\begin{equation}
\label{index Z(G)G'}
|G:Z(G)\overline{G'}|=|K:Z(G)\cap K|.
\end{equation}

(i)$\Rightarrow$(ii).
We are assuming that $G$ is a $C$-group.
Since $\overline{G'}$ and $K$ are abelian $C$-groups, they have finite exponent by \cref{torsion abelian}.
Hence both these subgroups decompose as a direct product of finitely many non-trivial Sylow subgroups:
$\overline{G'}=P_1 \times \cdots \times P_r$ and $K=Q_1 \times \cdots \times Q_s$.

Pick $P\in\{P_1,\ldots,P_r\}$ and $Q\in\{Q_1,\ldots,Q_s\}$ arbitrarily, and let us examine the
subgroup $Q\ltimes P$ of $G$.
By \cref{structure of closed abelian normal}, we can write $P=\Cr_{i\in I} \, \langle x_i \rangle$, where
each $\langle x_i \rangle$ is a normal subgroup of $G$ of prime order, say $p$.
As in the proof of \cref{torsion G' open}, we define the homomorphism $\theta_i:Q\longrightarrow \UU(\Z/p\Z)$
induced by conjugation of $Q$ on $\langle x_i \rangle$, and we consider the equivalence relation on $I$ given by
$i\sim j$ if and only if $\theta_i=\theta_j$.
Let $\CC$ be the set of equivalence classes of $\sim$ in $I$, and set
$P_C=\Cr_{i\in C} \, \langle x_i \rangle$.

For every $C\in\CC$, let us choose an index $i(C)\in C$.
Then we can write $P_C=\langle x_{i(C)} \rangle \times N_{i(C)}$, where
\[
N_{i(C)} = \Cr_{\substack{i\in C \\ i\ne i(C)}} \langle x_i \rangle \trianglelefteq G.
\]
Set $N=\Cr_{C\in\CC} \, N_{i(C)} \trianglelefteq_{c} G$.
Then $(Q\ltimes P)/N$ is naturally isomorphic to the semidirect product $S=Q\ltimes R$, where
$R=\Cr_{C\in\CC} \, \langle x_{i(C)} \rangle$,
and the action of $Q$ on each $\langle x_{i(C)} \rangle$ is given by $\theta_{i(C)}$.

We claim that the only cyclic subgroups of $R$ that are normal in $S$ are the $\langle x_{i(C)} \rangle$.
Indeed, assume that $\langle y \rangle\trianglelefteq S$, with $y=(y_C)_{C\in\CC}$ and $y_C\in \langle x_{i(C)} \rangle$.
If two entries in this tuple, say $y_C$ and $y_D$, are different from $1$, then since $C$ and $D$ are different
equivalence classes for $\sim$, there exists $z\in Q$ such that $y_C^z=y_C^m$, $y_D^z=y_D^n$ and $m\not\equiv n\pmod p$.
But then $y^z$ cannot be a power of $y$ and $\langle y \rangle \not\trianglelefteq S$.
This proves the claim.

Now since $S\cong (Q\ltimes P)/N$ is a $C$-group, its normal abelian subgroup $R$ is a direct product of cyclic normal subgroups
of order $p$ that are normal in $S$, by Lemma 3.1.7 and Lemma 3.2.3 of \cite{Sch}.
By the last paragraph, $R$ is the direct product of some of the subgroups $\langle x_{i(C)} \rangle$ with $C\in\CC$.
If $\CC$ is infinite then
\[
|R| \le |\Dr_{C\in\CC} \, \langle x_{i(C)} \rangle| \le |\CC| < 2^{|\CC|}
= |\Cr_{C\in\CC} \, \langle x_{i(C)} \rangle| = |R|,
\]
which is a contradiction.
Thus $\CC$ is necessarily finite.

It follows that
\[
C_Q(P) = \bigcap_{i\in I} \, \ker\theta_i = \bigcap_{C\in\CC} \, \ker\theta_{i(C)}
\]
has finite index in $Q$, since $Q/\ker\theta_{i(C)}$ embeds into $\UU(\Z/p\Z)$ and $\CC$ is finite.
Hence
\[
C_Q(\overline{G'}) = \bigcap_{i=1}^r \, C_Q(P_i)
\]
has also finite index in $Q$, and consequently
\[
C_K(\overline{G'}) = C_{Q_1}(\overline{G'}) \times \cdots \times C_{Q_s}(\overline{G'})
\]
has finite index in $K$.
Since $C_K(\overline{G'})=Z(G)\cap K$, this completes the proof of this implication, by \eqref{index Z(G)G'}.

(ii)$\Rightarrow$(i).
Set $L=Z(G)\cap K$.
Since $K$ is profinite-$C$, we can write $K=L\times M$ for some $M\le_{c} G$.
By hypothesis and \eqref{index Z(G)G'}, we have $|M|=|K:L|<\infty$.
Then we can write $G=H\times L$, where $H=M\ltimes \overline{G'}$.
Now note that $H$ is profinite-$C$ and torsion, and since $\overline{H'}$ coincides with $\overline{G'}$,
it follows that $\overline{H'}$ is open in $H$.
By \cref{torsion G' open}, $H$ is a $C$-group.
On the other hand, the torsion abelian profinite-$C$ group $L$ is a $C$-group by \cref{torsion abelian}.
Since the direct product of $C$-groups is a $C$-group, we conclude that $G$ is a $C$-group.
\end{proof}

We conclude with a remark about another property regarding the existence of complements. 
In the abstract setting, a group $G$ is said to be an $SC$-group if for every subgroup $H$ of $G$ there exists a subgroup $K$ of $G$ such that $J\cap K$ is a (not necessarily permutable) complement of $H$ in $J$ whenever $H\le J\le G$.
It is then natural to define a \emph{profinite-$SC$ group} as a profinite group $G$ in which the condition above
is fulfilled under the restriction that $H$ and $J$ should be closed in $G$.
Every $C$-group is an $SC$-group, but the converse does not hold, as is shown by the example of Tarski monsters.
Actually, Emaldi proved in \cite{E} that the class of $C$-groups coincides with that of locally finite $SC$-groups.
Now it is easy to see that a quotient of a profinite-$SC$ group by a closed normal subgroup is again profinite-$SC$.
In particular, a profinite-$SC$ group is an inverse limit of finite $SC$-groups, and then the following result
immediately follows.

\begin{proposition}
Let $G$ be a profinite group.
Then $G$ is a profinite-$SC$ group if and only if it is a profinite-$C$ group.
\end{proposition}

\subsection*{Acknowledgements} 
The second author wishes to thank the University of the Basque Country for its kind hospitality during the visit in which this work was initiated.

\end{document}